\documentclass[oneside,english]{amsart}
\usepackage[T1]{fontenc}
\usepackage[latin9]{inputenc}
\usepackage{geometry}
\usepackage{amsthm}
\usepackage{amssymb}
\usepackage{hyperref}
\usepackage[all]{xy}

\makeatletter
\numberwithin{equation}{section}
\numberwithin{figure}{section}
\theoremstyle{plain}
\newtheorem{thm}{\protect\theoremname}
  \theoremstyle{plain}
  \newtheorem{lem}[thm]{\protect\lemmaname}
  \theoremstyle{definition}
  \newtheorem{rem}[thm]{\protect\remarkname}
  \theoremstyle{plain}
  \newtheorem{prop}[thm]{\protect\propositionname}

\makeatother

\usepackage{babel}
  \providecommand{\lemmaname}{Lemma}
  \providecommand{\propositionname}{Proposition}
  \providecommand{\remarkname}{Remark}
\providecommand{\theoremname}{Theorem}

\DeclareMathOperator{\Ext}{Ext}
\DeclareMathOperator{\Sym}{Sym}
\DeclareMathOperator{\op}{op}

\DeclareMathOperator{\Aut}{Aut}
\DeclareMathOperator{\la}{la}

\begin{document}
\newcommand{\Zp}{{\mathbb{Z}_p}}
\newcommand{\Qp}{{\mathbb{Q}_p}}
\newcommand{\Fp}{{\mathbb{F}_p}}
\newcommand{\fr}[1]{\mathfrak{{#1}}}
\newcommand{\mb}{\mathbb}
\newcommand{\mc}{\mathcal}
\newcommand{\mf}{\mathfrak}
\title{A canonical dimension estimate for non-split semisimple p-adic Lie groups}
\author{Konstantin Ardakov and Christian Johansson}

\begin{abstract}
We prove that the canonical dimension of an admissible Banach space or a locally analytic representation of an arbitrary semisimple $p$-adic Lie group is either zero or at least half the dimension of a non-zero coadjoint orbit. This extends the results of Ardakov-Wadsley and Schmidt in the split semisimple case.
\end{abstract}
\maketitle

\section{Introduction}

This paper can be regarded as a postscript to \cite{AW1}. Its purpose
 is to record an argument that extends \cite[Theorem A]{AW1}
to compact semisimple $p$-adic analytic groups whose Lie algebra is not necessarily split
($p$ is a prime number). An analogue of \cite[Theorem A]{AW1} for distribution algebras was proved by Schmidt in \cite[Theorem 9.9]{Sch}; our argument works to extend this result as well. For simplicity, we mostly focus on the setting of Iwasawa algebras in this introduction. Let us recall that if $G$ is a compact
$p$-adic analytic group and $KG$ is its completed group ring (or
Iwasawa algebra) with respect to a finite field extension $K/\Qp$,
then the category $\mathcal{M}$ of finitely generated $KG$-modules
is an abelian category antiequivalent to the category of admissible
$K$-Banach space representations of $G$ (\cite{ST}). $\mathcal{M}$, and related
categories of $p$-adic representations of (locally) compact $p$-adic
analytic groups have received a lot of attention in the last decades,
motivated by research in the Langlands programme and Iwasawa theory.
Each object $M$ in $\mathcal{M}$ has a non-negative integer $d(M)=d_{KG}(M)$
attached to it called its canonical dimension. This notion
gives rise to a natural filtration of $\mathcal{M}$
\[
\mathcal{M}=\mathcal{M}_{d}\supseteq\mathcal{M}_{d-1}\supseteq...\supseteq\mathcal{M}_{0}
\]
by Serre subcategories, where $d=\dim\, G$ and $M\in\mathcal{M}_{i}$
if and only $d(M)\leq i$. The canonical dimension $d(M)$ may be
taken as a measure of the ``size'' of $M$; for example it is known
that $d(M)=0$ if and only $M$ is finite dimensional as $K$-vector
space, and $d(M)<d$ if and only if $M$ is torsion. It may also be regarded as a noncommutative analogue of the dimension of the support of a module in the commutative setting. The following result, together with its analogue for distribution algebras (Theorem \ref{thm: main2}) is the main result of this note:

\begin{thm}
\label{thm: main}Let $G$ be a compact $p$-adic analytic group whose
Lie algebra $\mathcal{L}(G)$ is semisimple, and let $G_{\mathbb{C}}$
denote a complex semisimple algebraic group with the same root system
as $\mathcal{L}(G)\otimes_{\Qp}\overline{\mathbb{Q}}_{p}$.
Let $M$ be a finitely generated $KG$-module which is infinite dimensional
as a $K$-vector space, and assume that $p$ is an odd very good prime
for $G$. Then $d_{KG}(M)\geq r$, where $r$ is half the smallest
possible dimension of a nonzero coadjoint $G_{\mathbb{C}}$-orbit.
\end{thm}

For the definition of a very good prime see \cite[\S 6.8]{AW1}. We
recall that Theorem A of \cite{AW1} proves the same conclusion under
the additional hypothesis that $\mathcal{L}(G)$ is \emph{split} semisimple.
Since the Langlands programme often deals with groups that are not
split, it seemed natural to wish to remove this assumption from the
theorem. Examples of such nonsplit groups are $G=\mathbb{G}(K)$, where $\mathbb{G}/K$ is an anisotropic semisimple algebraic group. 

The proof is a variation of that of the split case. Let us first give a
short rough description of the split case and refer to the introduction of \cite{AW1}
for more details. A short argument reduces the problem to the
case when $G$ is a uniform pro-$p$ group. This implies we may associate
to $G$ a certain $\Zp$-Lie algebra $\mathfrak{h}$. The algebra $KG$ comes with a morphism into an inverse system $(D_{w}=D_{p^{-1/w}}(G,K))_{w=p^{n}}$
of distribution algebras of $G$, whose inverse limit is the locally
analytic distribution algebra $D(G,K)$ with coefficients in $K$
studied in detail by Schneider and Teitelbaum. For large enough $w=p^{n}$,
$D_{w}\otimes_{KG}M\neq0$. The algebra $D_{w}$ turns out to be a
crossed product $\mathcal{U}_{n}\ast G/G^{p^{n}}$, where $\mathcal{U}_{n}=\widehat{U(p^{n}\mathfrak{h})}_{K}$
is an affinoid enveloping algebra, and it turns out that it is enough
to prove the analogous result for the algebras $\mathcal{U}_{n}$.
The proof of this analogue relies on some technical work: analogues
of Beilinson-Bernstein localisation, Quillen's Lemma and Bernstein's
Inequality. The assumption that $\mathcal{L}(G)$ is split comes up
in the localisation theory, which works for affinoid enveloping algebras
of $\mathcal{O}_{L}$-Lie algebras that are of the form $\pi^{m}\mathfrak{g}$,
with $L/\Qp$ a finite extension, $\pi$ a uniformizer
in $L$ and $\mathfrak{g}$ a split Lie algebra over $\mathcal{O}_{L}$. 

In this paper we follow the same reduction steps to get nonzero $\mathcal{U}_{n}$-modules
$D_{p^{n}}\otimes_{KG}M$ for sufficiently large $n$. Since the techniques of \cite{AW1} are not strong enough to deduce
a canonical dimension estimate for the affinoid enveloping algebra $\mathcal{U}_{n}$ in
this case, we base change to a finite extension $L/\Qp$
such that the Lie algebra $\mathcal{L}(G)\otimes_{\Qp}L$ is split and
hence has a split lattice $\mathfrak{g}$ in addition to the lattices
$p^{n}\mathfrak{h}\otimes_{\Zp}\mathcal{O}_{L}$. We then
sandwich a $p$-power multiple of $\mathfrak{g}$ in between two sufficiently
large $p$-power multiples of $\mathfrak{h}\otimes_{\Zp}\mathcal{O}_{L}$
in order to base change one of the $D_{p^{n}}\otimes_{KG}M$ to an affinoid enveloping algebra for which the theory of \cite{AW1} applies. 

As remarked earlier, Schmidt gave an analogue of \cite[Theorem A]{AW1} for coadmissible modules of the distribution algebra $D(G,K)$, using the methods of \cite{AW1}. Here we may now allow $G$ to be a locally $L$-analytic group for some intermediate extension $K/L/\Qp$. Our argument works equally well in this setting to remove the splitness hypothesis of \cite[Theorem 9.9]{Sch}, giving us the following theorem:

\begin{thm}
\label{thm: main2}Let $G$ be a locally $L$-analytic group whose
Lie algebra $\mathcal{L}(G)$ is semisimple, and let $G_{\mathbb{C}}$
denote a complex semisimple algebraic group with the same root system
as $\mathcal{L}(G)\otimes_{L}\overline{\mathbb{Q}}_{p}$.
Let $M$ be a coadmissible $D(G,K)$-module such that $d_{D(G,K)}(M)\geq 1$ and assume that $p$ is an odd very good prime
for $G$. Then $d_{D(G,K)}(M)\geq r$, where $r$ is half the smallest
possible dimension of a nonzero coadjoint $G_{\mathbb{C}}$-orbit.
\end{thm}

We remark that Theorem \ref{thm: main} follows from Theorem \ref{thm: main2} by the faithful flatness of the distribution algebra over the Iwasawa algebra (\cite[Theorem 5.2]{ST2}). Nevertheless we have opted to give independent arguments, keeping the proof of Theorem \ref{thm: main} independent of \cite[Corollary 5.13]{Sch} (in the case $L=\Qp$), which was conjectured in \cite{AW1} but neither proved nor needed. We also remark that, for a coadmissible $D(G,K)$-module $M$, $d_{D(G,K)}(M)=0$ if and only if $D_{w}\otimes_{D(G,K)}M$ is a finite-dimensional $K$-vector space for all $w$.

Let us now describe the contents of this paper. Section \ref{sec: flatness}
proves the necessary flatness results needed make the sandwich argument
work and recalls some generalities on crossed products. The flatness
result is a consequence of a general result of Berthelot and Emerton.
Section \ref{sec: main} introduces the main players, discusses some
complements to results in \cite{AW1}, and then proves Theorem \ref{thm: main},
fleshing out the strategy described above. Finally, section 4 proves Theorem \ref{thm: main2}.

The second author was supported by National Science Foundation grants 093207800 and DMS-1128155 during the work on this paper whilst a member of MSRI and the IAS respectively, and would like to thank both institutions for their hospitality.  The first author was supported by EPSRC grant EP/L005190/1.

\section{Flatness and Crossed Products\label{sec: flatness}}

Fix a prime number $p$ and let $K$ be a finite extension of $\Qp$.
In this section we give a condition for the natural map between affinoid
enveloping algebras induced from an inclusion of finite free $\mathcal{O}_{K}$-
Lie algebras of the same rank to be left and right flat. All completions
arising in this paper are $p$-adic completions. We use a technique
due to Berthelot, abstracted by Emerton in  \cite[Proposition 5.3.10]{Em}. 

\begin{lem}

\label{lem: flatness}Let $\mathfrak{h}\subseteq\mathfrak{g}$ be
$\mathcal{O}_{K}$-Lie algebras, both finite free of rank $d$. Assume
that $[\mathfrak{g},\mathfrak{h}]\subseteq\mathfrak{h}$. Then \textup{$\widehat{U(\mathfrak{g})}_{K}$
is a flat left and right $\widehat{U(\mathfrak{h})}_{K}$-module.}

\end{lem}

\begin{proof}

The proof is very close to that of \cite[Proposition 3.4]{SS}, which
contains the case $\mathfrak{h}=p\mathfrak{g}$ as a special case.
First we prove right flatness. Put $A=U(\mathfrak{h})$ and $B=U(\mathfrak{g})$
and let $F_{\bullet}A$ resp. $F_{\bullet}B$ denote the usual Poincar\'e-Birkhoff-Witt
filtrations on $A$ and $B$ respectively. $A$ and $B$ are well known to be
$p$-torsion-free and $p$-adically separated left Noetherian $\Zp$-algebras.
Define a new increasing filtration on $B$ by the $\Zp$-submodules

\[
F_{i}^{\prime}B=A\cdot F_{i}B = \sum_{j\leq i} A \fr{g}^j.
\]

In order to apply \cite[Proposition 5.3.10]{Em}, we need to verify
conditions (i)-(iii) in the statement of \cite[Lemma 5.3.9]{Em}.
Condition (ii), that $F_{0}^{\prime}B=A$, is clear. Condition (i)
asks that $F_{i}^{\prime}B\cdot F_{j}^{\prime}B\subseteq F_{i+j}^{\prime}B$. 
To see this, note that because $\fr{g}$ normalises $\fr{h}$, we have $[\fr{g}, A] \subseteq A$, so $\fr{g} A \subseteq A \fr{g} + A$. Hence $\fr{g}^i A \subset F_i^\prime B$ by an easy induction, so $A \fr{g}^i A \fr{g}^j \subset F_{i+j}^\prime B$ as required. Condition (iii) asks that $Gr_{\bullet}^{F^{\prime}}B$
is finitely generated over $A$ by central elements. Since $F_{\bullet}B\subseteq F_{\bullet}^{\prime}B$,
the identity map induces a natural map $f\,:\, Gr_{\bullet}^{F}B\rightarrow Gr_{\bullet}^{F^{\prime}}B$
and since $Gr_{0}^{F^{\prime}}B=A$ we see that $Gr_{\bullet}^{F^{\prime}}B$
is generated (as a ring) by $f(Gr_{\bullet}^{F}B)$ and $A$. Now
$Gr_{\bullet}^{F}B=\Sym(\mathfrak{g})$, thus we may find $X_{1},...,X_{d}\in\mathfrak{g}$
whose symbols $gr^{F^{\prime}}X_{k}\in Gr_{1}^{F^{\prime}}B$ generate
$Gr_{\bullet}^{F^{\prime}}B$ over $A$. Thus we have shown finite
generation and to show that the $gr^{F^{\prime}}X_{k}$ are central
it suffices to prove that for any $X\in\mathfrak{g}$, $gr^{F^{\prime}}X$
commutes with $A$ in $Gr_{\bullet}^{F^{\prime}}B$. Since $A$ is
generated by $\mathfrak{h}$ it suffices verify this for $Y\in\mathfrak{h}\subseteq Gr_{0}^{F^{\prime}}B$.
Computing, we see that $(gr^{F^\prime}\, X)(gr^{F^\prime}\, Y)-(gr^{F^\prime}\, Y)(gr^{F^\prime}\, X)$ is equal
to the image of the bracket $[X,Y]$ inside $Gr_{1}^{F^{\prime}}B$.
By assumption $[X,Y]\in\mathfrak{h}$, so it becomes $0$ in $Gr_{1}^{F^{\prime}}B$.
Having verified the assumptions of \cite[Proposition 5.3.10]{Em}
we may conclude that $\widehat{B}_K$ is a flat right $\widehat{A}_K$-module.

To prove left flatness we remark that it is well known that $U(\mathfrak{h})^{\op}\rightarrow U(\mathfrak{g})^{\op}$
is canonically isomorphic to $U(\mathfrak{h}^{\op})\rightarrow U(\mathfrak{g}^{\op})$,
where we recall that the opposite Lie algebra $\mathfrak{h}^{\op}$
of $\mathfrak{h}$ is the $\mathcal{O}_{K}$-module $\mathfrak{h}$
together with a new bracket $[-,-]^{\prime}$ defined by $[X,Y]^{\prime}=[Y,X]$.
Thus, having proved right flatness of $U(\mathfrak{h}^{\op})\rightarrow U(\mathfrak{g}^{\op})$
above we may deduce left flatness of $U(\mathfrak{h})\rightarrow U(\mathfrak{g})$.
\end{proof}

Next let us recall the notion of a crossed product from e.g. \cite[\S 1]{Pa}. Let $S$ be a ring and let $H$ be a group. A \emph{crossed product}
$S\ast H$ is a ring containing $S$ as a subring and a subset $\overline{H}=\{\bar{h}\mid h\in H\}$
of units bijective with $H$, that satisfies the following conditions:

\begin{itemize}

\item $S\ast H$ is a free right $S$-module with basis $\overline{H}$.

\item We have $\bar{h}S=S\bar{h}$ and $\overline{gh}S=\bar{g}\bar{h}S$
for all $g,h\in H$.
\end{itemize}

Given a crossed product $S\ast H$ we obtain functions $\sigma\,:\, H\rightarrow \Aut(S)$
and $\tau\,:\, H\times H\rightarrow S^{\times}$, called the action
and the twisting, defined by 
\[
\sigma(h)(s):=(\bar{h})^{-1}s\bar{h},
\]
\[
\tau(g,h)=(\overline{gh})^{-1}\bar{g}\bar{h}.
\]

The associativity of $S\ast H$ is equivalent to certain relations between
$\sigma$ and $\tau$ (\cite[Lemma 1.1]{Pa} ) and conversely, given
$S$, $H$, $\sigma$ and $\tau$ satisfying these relations we may
construct $S\ast H$ as the free right $S$-module on a set $\overline{H} = \{\bar{h} : h\in H\}$ which is in bijection with $H$, and multiplication being defined by extending the rule 
\[
(\bar{g} r)(\bar{h} s)= \overline{gh}  \tau(g,h) r^{\sigma(h)}s
\]
additively ($g,h\in H$, $r,s\in S$). Here we use the notation $s^{f}$ to denote the (right) action of $f\in \Aut(S)$ on $s\in S$. Let us record how crossed products
behave with respect to taking opposites.

\begin{lem}

\label{lem: crossed}Let $S\ast H$ be a crossed product of ring $S$
by a group $H$, with action $\sigma\,:\, H\rightarrow \Aut(S)$ and
twisting $\tau\,:\, H\times H\rightarrow S^{\times}$.

1) $(S\ast H)^{\op}=S^{\op}\ast H^{\op}$, with action $\sigma^{\op}$
and twisting $\tau^{\op}$ on the right hand side given by 
\[
\sigma^{\op}(h)=\sigma(h)^{-1},
\]
\[
\tau^{\op}(g,h)=\tau(h,g)^{\sigma(g)^{-1}\sigma(h)^{-1}}.
\]

2) Let $T$ be another ring with a homomorphism $\phi\,:\, S\rightarrow T$ and let $\Gamma\subseteq Aut(S)$ be a subgroup of the automorphisms of $S$ that contains the inner automorphisms defined by $\tau(g,h)$ for $g,h\in H$ and the $\sigma(h)$ for $h\in H$. We assume that there is a compatible homomorphism $\psi\,:\, \Gamma \rightarrow \Aut(T)$, where by compatible, we mean that for $f\in \Gamma$, $\phi\circ f=\psi(f)\circ\phi$, and additionally that $\psi$ sends the inner automorphism defined by $\tau(g,h)\in S^{\times}$ to the inner automorphism defined by $\phi(\tau(g,h))$ for all $g,h\in H$. Then we may form a crossed
product $T\ast H$ with action $\psi\circ\sigma$ and twisting $\phi\circ\tau$.
Moreover $\phi$ extends naturally to a homomorphism $\Phi\,:\, S\ast H\rightarrow T\ast H$.

3) The constructions in 1) and 2) are compatible: Given the situation in 2), we also have $\phi\, : S^{\op} \rightarrow T^{\op}$ and $S^{\op} \ast H^{\op}$; we may form $T^{\op} \ast H^{\op}$ and we obtain a natural map $S^{\op} \ast H^{\op} \rightarrow T^{\op} \ast H^{\op}$. Under the identification in 1) this agrees with $(T\ast H)^{\op}$ and $\Phi$.

\end{lem}

\begin{proof}

1) $(S\ast H)^{\op}$ contains $S^{\op}$ and as a subring as well as the set $\overline{H}$ which is bijective with $H^{\op}$.  To verify that $(S\ast H)^{\op}$ is a crossed product $S^{\op} \ast H^{\op}$ it remains to verify that, working in $S \ast H$, $S\bar{h}$ is a free left $S$-module and that $S\overline{gh}=(S\bar{g})\bar{h}$ for all $g,h\in H$. The first assertion follows from $s\bar{h}=\bar{h}s^{\sigma(h)}$ and that $\bar{h}S$ is a free right $S$-module and $\sigma(h)$ is an automorphism of $S$. The second follows similarly using the formula
$$ \bar{g}\bar{h} = \tau(g,h)^{\sigma(gh)^{-1}}\overline{gh}. $$
We may then compute the formulae for $\sigma^{\op}$ and $\tau^{\op}$, using $\cdot$ to distinguish the multiplication in $(S\ast H)^{\op}$ or $H^{\op}$ from that in $S\ast H$ or $H$:
$$s^{\sigma^{\op}(h)}=(\bar{h})^{-1} \cdot s \cdot \bar{h} = \bar{h}s(\bar{h})^{-1} =s^{\sigma(h)^{-1}}; $$
$$ \tau^{\op}(g,h)= (\overline{g\cdot h})^{-1}\cdot \bar{g} \cdot \bar{h}=\bar{h}\bar{g}(\overline{hg})^{-1}= \tau(h,g)^{\sigma(g)^{-1}\sigma(h)^{-1}}. $$
This finishes the proof of (1). We now prove (2). First, it is straightforward to verify that $\psi \circ \sigma$ and $\phi \circ \tau$ satisfy  \cite[Lemma 1.1]{Pa} using the compatibility condition. Thus we may form $T\ast H$ as above, and $\phi$ induces an additive group homomorphism $\Phi\, :\, S\ast H \rightarrow T\ast H$ given by $\Phi(\bar{h}s)=\bar{h}\phi(s)$ which is easily checked to be a ring homomorphism.  Finally, to check 3) we note that
$$ (\psi \circ \sigma)^{\op}(h)=\psi(\sigma(h))^{-1}=\psi(\sigma(h)^{-1})=\psi(\sigma^{\op}(h))$$
and
$$(\phi \circ \tau)^{\op}(g,h)= \phi(\tau(h,g))^{\psi(\sigma(g))^{-1}\psi(\sigma(h))^{-1}}= \phi(\tau(h,g))^{\psi(\sigma(g)^{-1}\sigma(h)^{-1})}= $$
$$=\phi\left( \tau(h,g)^{\sigma(g)^{-1}\sigma(h)^{-1}}\right)=\phi(\tau^{\op}(g,h)). $$
This shows that $(T\ast H)^{\op} =T^{\op} \ast H^{\op}$, and checking that the maps agree is another straightforward computation.
\end{proof}

\begin{rem}

\label{AB}This Lemma implies that the natural left module analogue
of \cite[Lemma 5.4]{AB}  holds (this may of course also be proved directly).
References to \cite[Lemma 5.4]{AB} below will often implicitly be
to its left module analogue.
\end{rem}

\section{Canonical Dimensions for Iwasawa algebras\label{sec: main}}

Let us first fix some notation and terminology. As in the previous
section we let $K$ be a finite field extension of $\Qp$.
In this section we let $G$ denote a compact $p$-adic analytic group.
For the definition, see \cite[\S 8]{DDMS} ; note that for us a $p$-adic
analytic group is the same as a $\Qp$-analytic group.
Put $d=\dim\, G$ (this is the dimension of $G$ as a $\Qp$-analytic
group, so e.g. $\dim\, SL_{2}(K)=3[K:\Qp]$). As in \cite[\S 9.5]{DDMS}
 we let $\mathcal{L}(G)$ denote the Lie algebra of $G$; this
is a $\Qp$-Lie algebra of dimension $d$. We let $KG$
denote the Iwasawa algebra of $G$ with coefficients in $K$, defined
as $KG=\left(\underset{\longleftarrow}{\lim}\, \mathcal{O}_{K}[G/N]\right)\otimes_{\mathcal{O}_{K}}K$,
where the inverse limit runs through all open normal subgroups $N\subseteq G$.
$KG$ is Auslander-Gorenstein with self-injective dimension $d$ (see
\cite[Definition 2.5]{AW1}  for this notion). This follows for example
from \cite[Lemma 5.4]{AB} , \cite[Lemma 10.13]{AW1}  and the existence
of a uniform normal open subgroup of $G$ (see below for this notion).
For any Auslander-Gorenstein ring $A$, we define the grade $j_{A}(M)$
and canonical dimension $d_{A}(M)$ of any nonzero finitely generated
$A$-module $M$ by
\[
j_{A}(M)= \min\,\left\{ j\mid \Ext_{A}^{j}(M,A)\neq0\right\} ,
\]
\[
d_{A}(M)={\rm inj}\, {\rm dim}_A A-j_{A}(M).
\]

See \cite[\S 2.5]{AW1}  for more definitions and details. We will
also use the notion of a \emph{uniform pro-$p$ group}, for which we refer
to \cite[\S 4 and \S 8.3]{DDMS}. A uniform pro-$p$ group $H$ has
an associated $\Zp$-Lie algebra $L_{H}$ --- see \cite[\S 4.5]{DDMS}
 --- which is free of rank $\dim\, H$. We also define $\mathcal{L}(H)=L_{H}\otimes_{\Zp}\Qp$.
The remainder of this section will be devoted to the proof of Theorem
\ref{thm: main}. First, we record a simple reduction:

\begin{lem}

It suffices to prove Theorem \ref{thm: main} for uniform pro-$p$
groups and for $K=\Qp$.
\end{lem}

\begin{proof}

Pick a uniform normal open subgroup $H$ of $G$. Let $M$ be a finitely
generated $KG$-module. Then $\mathcal{L}(G)=\mathcal{L}(H)$, so
$r_{G}=r_{H}$, and \cite[Lemma 5.4]{AB}  implies that $d_{KG}(M)=d_{KH}(M)$.
Hence we may reduce to uniform pro-$p$ groups. Then \cite[Lemma
2.6]{AW1}  implies that it suffices to prove it for $K=\Qp$.
\end{proof}

Before we proceed we need to recall the rings $\mathcal{U}_{n}\ast H_{n}$,
introduced in \cite[\S 10.6]{AW1} . From now on let us assume that
$G$ is uniform and $K=\Qp$. We may then set $\mathfrak{h}=p^{-1}L_{G}$;
since $L_{G}$ satisfies $[L_{G},L_{G}]\subseteq pL_{G}$, $\mathfrak{h}$
is a $\Zp$-sub Lie algebra of $\mathcal{L}(G)$, finite
free of rank $d$. Recall that, for all $n\in\mathbb{Z}_{\geq0}$,
the groups $G^{p^{n}}$ are normal uniform pro-$p$ subgroups of $G$ with $L_{G^{p^{n}}}=p^{n}L_{G}$
and we set $H_{n}=G/G^{p^{n}}$. In \cite[ \S 10.6]{AW1}, $\mathcal{U}_{n}$
is defined to be the microlocalisation of $\Zp G^{p^{n}}$ with respect
to the microlocal Ore set $S_{n}=\bigcup_{a\geq0}p^{a}+\mathfrak{m}_{n}^{a+1}$,
where $\mathfrak{m}_{n}$ is the unique (left and right) maximal ideal
of $\Zp G^{p^{n}}$. If we need to emphasize the group $G$ we will write
$\mathcal{U}_{n}(G)$. By \cite[Theorem 10.4]{AW1},  $\mathcal{U}_{n}$
is isomorphic to $\widehat{U(p^{n}\mathfrak{h})}_{\Qp}$.
$\Zp G$ is a crossed product $\Zp G^{p^{n}}\ast H_{n}$ and since $S_{n}$
is invariant under automorphisms we obtain a canonical homomorphism
$\Aut(\Zp G^{p^{n}})\rightarrow \Aut(\mathcal{U}_{n})$. Note that if $f_{1},f_{2}\in \Aut(\mathcal{U}_{n})$
come from $\Aut(\Zp G^{p^{n}})$ and agree on $\Zp G^{p^{n}}$ then they
are equal (this is immediate from the construction). Thus we may form
$\mathcal{U}_{n}\ast H_{n}$ by Lemma \ref{lem: crossed}(2) and we
get an induced homomorphism $\Zp G\rightarrow\mathcal{U}_{n}\ast H_{n}$. Since $p$ is invertible in $\mathcal{U}_{n}\ast H_{n}$, we have a homomorphism $\Qp G\rightarrow\mathcal{U}_{n}\ast H_{n}$.
We will need some left/right complements to various results in \cite{AW1}:

\begin{prop}

\label{cor: roundup}1) The natural map $\Qp G\rightarrow\mathcal{U}_{n}\ast H_{n}$
is left and right flat. 

2) If $M$ is a finitely generated left $\Qp G$-module, then $(\mathcal{U}_{n}\ast H_{n})\otimes_{\Qp G}M=0$
if and only if $M$ is $S_{n}$-torsion.

3) If $M$ is a finitely generated right $\Qp G$-module, then $M\otimes_{\Qp G}(\mathcal{U}_{n}\ast H_{n})=0$
if and only if $M$ is $S_{n}$-torsion.

4) If $M$ is a finitely generated $p$-torsion free left or right
$\Zp G$-module, then there exists an $n_{0}\in\mathbb{Z}_{\ge0}$ such
that $M$ is $S_{n}$-torsion-free for all $n\geq n_{0}$.

\end{prop}

\begin{proof}

1) Right flatness is \cite[Proposition 10.6(d)]{AW1}  and from this
we also get right flatness of $\Qp G^{\op}\rightarrow\mathcal{U}_{n}(G^{\op})\ast H_{n}^{\op}$
(i.e. performing the same constructions for the opposite group). However
we have identifications $\Qp G^{\op}=(\Qp G)^{\op}$ and $\mathcal{U}_{n}(G^{\op})\ast H_{n}^{\op}=(\mathcal{U}_{n}(G)\ast H_{n})^{\op}$
identifying $\Qp G^{\op}\rightarrow\mathcal{U}_{n}(G^{\op})\ast H_{n}^{\op}$
with $(\Qp G)^{\op}\rightarrow(\mathcal{U}_{n}(G)\ast H_{n})^{\op}$ using
Lemma \ref{lem: crossed}, so we get the desired left flatness.

2) is \cite[Proposition 10.6(e)]{AW1}  and 3) follows from 2) applied
to $G^{\op}$ with the identifications in the proof of 1). Similarly
the left part of 4) is \cite[Corollary 10.11]{AW1}  and the right part
follows as above.
\end{proof}

Before we get to the proof of the main theorem we will abstract a
short calculation from the proof:

\begin{lem}

\label{lem: sandwich}Let $A$, $B$, $S$ and $T$ be rings and assume
that we have a commutative diagram
\[
\xymatrix{A\ar[r]\ar[d] & B\ar[d]\\
S\ar[r] & T
}
\]
where $A\rightarrow S$ makes $S$ into a crossed product $A\ast G$
for some finite group $G$. Let $M$ be a finitely generated left
$S$-module and assume that $\Ext_{S}^{i}(M,S)\otimes_{S}T\neq0$ for
some $i$. Then 
\[\Ext_{A}^{i}(M,A)\otimes_{A}B\neq0.\]

\end{lem}

\begin{proof}

Assume that $\Ext_{A}^{i}(M,A)\otimes_{A}B=0$. Then $\Ext_{A}^{i}(M,A)\otimes_{A}T=0$.
However, by Remark \ref{AB} $\Ext_{A}^{i}(M,A)=\Ext_{S}^{i}(M,S)$
as right $A$-modules. Thus $\Ext_{S}^{i}(M,S)\otimes_{A}T=0$. But
$\Ext_{S}^{i}(M,S)\otimes_{A}T$ surjects onto $\Ext_{S}^{i}(M,S)\otimes_{S}T$,
a contradiction.
\end{proof}

We now come to the proof of Theorem \ref{thm: main}. Recall from
\cite[Theorem 3.3, Proposition 9.1(b)]{AW1}  and the paragraph below
it that $\mathcal{U}_{n}$ is Auslander-Gorenstein with self-injective
dimension $d$ (in fact it is Auslander regular). By \cite[Lemma
5.4]{AB}  $\mathcal{U}_{n}\ast H_{n}$ is also Auslander-Gorenstein of self-injective
dimension $d$ ($\mathcal{U}_{n}\ast H_{n}$ is also Auslander regular).
We will use this freely in the proof (in particular the fact that
the self-injective dimensions agree). 

\begin{prop}
\label{prop: main}
Theorem \ref{thm: main} is true when $G$ is uniform and $K=\Qp$.

\end{prop}

\begin{proof}

Let $M$ be a finitely generated $\Qp G$-module which is
not finite dimensional as a $\Qp$-vector space. Let $j=j_{\Qp G}(M)$
and set $N:=\Ext_{\Qp G}^{j}(M,\Qp G)$, this is
a finitely generated right $\Qp G$-module. By Proposition \ref{cor: roundup} we can find $t \geq 0$ such that $N$ is $S_n$-torsion-free for all $n\geq t$. Now $N = \Ext_{\Qp G^{p^n}}^j (M, \Qp G^{p^n}$) and $d_{\Qp G}(M) = d_{\Qp G^{p^n}}(M)$ for any $n \geq 0$  by \cite[Lemma 5.4]{AB}, so by replacing $G$ by $G^{p^t}$ if necessary we may assume $t = 0$. Thus we can assume that $N$ is $S_0$-torsion-free.

Let $F$ be a finite extension of $\Qp$ such that $\mathcal{L}(G)\otimes_{\Qp}F$
is a split $F$-Lie algebra. Then we may find a split semisimple simply
connected algebraic group $\mathbb{G}/\mathcal{O}_{F}$ whose $\mathcal{O}_{F}$-Lie
algebra $\mathfrak{g}$ satisfies $\mathfrak{g}\otimes_{\mathcal{O}_{F}}F\cong\mathcal{L}(G)\otimes_{\Qp}F$.
We fix this isomorphism throughout and consider $\mathfrak{g}$ as
a subset of $\mathfrak{g}\otimes_{\mathcal{O}_{F}}F$. 
Recall that $\fr{h} = p^{-1}L_G$; we can find integers $n, m \geq 0$ such that
  \[ p^{n}(\mathfrak{h}\otimes_{\Zp}\mathcal{O}_{F})\subseteq p^{m}\mathfrak{g} \subseteq \mathfrak{h}\otimes_{\Zp}\mathcal{O}_{F}.\]
Now because $N$ is $S_0$-torsion-free, $N \otimes_{\Qp G} \mathcal{U}_0$ is non-zero, and therefore $N\otimes_{\Qp G}(\mathcal{U}_{n}\ast H_{n})$ is also non-zero. By applying Proposition \ref{cor: roundup} together with \cite[Proposition 2.6]{AW1} we deduce that 
\[
d_{\Qp G}(M)=d_{\mathcal{U}_0}(M_0)=d_{\mathcal{U}_{n}\ast H_{n}}(M_{n})
\]
where $M_{n}:=(\mathcal{U}_{n}\ast H_{n})\otimes_{\Qp G}M$. We may now again apply \cite[Lemma 5.4]{AB}  to deduce that 
\[
d_{\mathcal{U}_{n}\ast H_{n}}(M_{n})=d_{\mathcal{U}_{n}}(M_{n}).
\]
Furthermore \cite[Lemmas 2.6 and 3.9]{AW1}  give us that 
\[  d_{\mathcal{U}_{n}}(M_{n})=d_{F\otimes_{\Qp}\mathcal{U}_{n}}(F\otimes_{\Qp}M) \quad\mbox{and} \quad F\otimes_{\Qp}\mathcal{U}_{n}=\widehat{U(p^{n}\mathfrak{h}\otimes_{\Zp}\mathcal{O}_{F})}_{F}.\]
Since $F\otimes_{\Qp}\mathcal{U}_{n}\rightarrow F\otimes_{\Qp}\mathcal{U}_0$ factors through $\widehat{U(p^{m}\mathfrak{g})}_{F}$, there is a commutative
diagram 
\[
\xymatrix{F\otimes_{\Qp}\mathcal{U}_{n}\ar[r]\ar[d] & \widehat{U(p^{m}\mathfrak{g})}_{F}\ar[d]\\
(F\otimes_{\Qp}\mathcal{U}_{n})\ast H_{n}\ar[r] & F\otimes_{\Qp}\mathcal{U}_0.
}
\]
Let $S = (F \otimes_{\Qp} \mathcal{U}_n) \ast H_n$ and $T = F \otimes_{\Qp} \mathcal{U}_0$, and note that $F \otimes_{\Qp} M_n = S \otimes_{\Qp G} M$, so that 
\[\Ext^j_S(F \otimes_{\Qp} M_n, S) \otimes_S T \cong \Ext^j_S(S \otimes_{\Qp G}M, S) \otimes_S T \cong \Ext^j_{\Qp G}(M, \Qp G) \otimes_{\Qp G} T = N \otimes_{\Qp G} T \neq 0\]
because $N \otimes_{\Qp G} \mathcal{U}_0 \neq 0$ by construction.  Hence Lemma \ref{lem: sandwich} implies that
\[
\Ext^{j}_{F\otimes_{\Qp}\mathcal{U}_{n}}(F\otimes_{\Qp}M_n,F\otimes_{\Qp}\mathcal{U}_{n})\otimes_{F\otimes_{\Qp}\mathcal{U}_{n}}\widehat{U(p^{m}\mathfrak{g})}_{F}\neq0.
\]
Since $p^{n}\mathfrak{h}\otimes_{\Zp}\mathcal{O}_{F}\subseteq p^{m}\mathfrak{g}$ and $[p^{m}\mathfrak{g},p^{n}\mathfrak{h}\otimes_{\Zp}\mathcal{O}_{F}]\subseteq p^{n}\mathfrak{h}\otimes_{\Zp}\mathcal{O}_{F}$,
Lemma \ref{lem: flatness} implies that the top arrow in the above commutative diagram $F\otimes_{\Qp}\mathcal{U}_{n}\rightarrow\widehat{U(p^{m}\mathfrak{g})}_{F}$
is left and right flat. Thus \cite[Proposition 2.6]{AW1} applied to the top arrow in the commutative diagram above gives
\[d_{\Qp G}(M)=d_{F \otimes_{\Qp} \mathcal{U}_n}(F \otimes_{\Qp} M_n) = d_{\widehat{U(p^{m}\mathfrak{g})}_{F}}(V)\]
where $V := \widehat{U(p^m \fr{g})}_F \otimes_{F \otimes_{\Qp} \mathcal{U}_n} (F \otimes_{\Qp} M_n)$. Here we have used the fact that $\widehat{U(p^{m}\mathfrak{g})}_{F}$ is Auslander regular
of self-injective dimension $d$. Now \cite[Lemma 10.13]{AW1}  implies
that $d_{\Qp G}(M)\geq1$ since $M$ is not finite dimensional
as a $\Qp$-vector space (we note here that 
of \cite[Lemma 10.13]{AW1} does not require $G$ to satisfy the assumptions of
 of \cite[\S 10.12]{AW1}). Finally, \cite[Theorem 9.10]{AW1}  implies
that $d_{\widehat{U(p^{m}\mathfrak{g})}_{F}}(V)\geq r$, so $d_{\Qp G}(M)\geq r$ as desired.
\end{proof}

\section{Canonical dimensions for distribution algebras}

In this section we let $L\subseteq K$ be finite extensions of $\Qp$ and let $G$ be a locally $L$-analytic group. We let $D(G,K)$ denote the algebra of $K$-valued distributions on $G$, studied in depth in \cite{ST2}, which we refer to for more details and some terminology.

Following the notation of \cite{ST2} we let $G_{0}$ denote the underlying $\Qp$-analytic group obtained from $G$ by forgetting the $L$-structure. Recall (see e.g. \cite[\S 5.6]{Sch}) that $G$ is said to be $L$-uniform if $G_0$ is uniform and $L_{G_{0}}\subseteq \mc{L}(G)$ is an $\mc{O}_{L}$-lattice. If this holds we will write $L_{G}$ for $L_{G_{0}}$ with its $\mc{O}_{L}$-module structure. When $G$ is $L$-uniform, $D(G,K)$ carries a Fr\'echet-Stein structure given by the inverse system $(D_{r}(G,K))_{r\in [p^{-1},1)\cap p^{\mathbb{Q}}}$ (see e.g. \cite[\S 5.17]{Sch}). We put $r_{n}=p^{-1/p^{n}}$ and will from now on only consider the inverse system $(D_{r_{n}}(G,K))_{n\in \mathbb{Z}_{\geq 0}}$, which is final in the previous inverse system. The rings $D_{r_{n}}(G,K)$ are Auslander regular of self-injective and global dimension $d=\dim_{L}G$ when $n\geq 1$ by \cite[Proposition 9.3]{Sch}. Thus, the category of coadmissible $D(G,K)$-modules (\cite[\S 6]{ST2}) has a well defined dimension theory (\cite[\S 8]{ST2}). Note that if $G$ is $L$-uniform then so is $G^{p^{n}}$ for all $n\geq 0$, and we have $D_{r_{0}}(G^{p^{n}},K)\ast H_{n}= D_{r^{n}}(G,K)$ by \cite[Corollary 5.13]{Sch} with $H_{n}=G/G^{p^{n}}$ as in the previous section.

When $G$ is a general locally $L$-analytic group, $G$ has at least one (and hence many) open $L$-uniform subgroup(s); indeed any compact open locally $L$-analytic subgroup of $G$ has a basis of neighbourhoods of the identity consisting of open normal $L$-uniform subgroups by \cite[Corollary 4.4]{Sch1}. Thus the category of coadmissible $D(G,K)$-modules has a well defined dimension theory using the formalism of \cite[\S 8]{ST2}, defined by restriction to an arbitrary open $L$-uniform subgroup. Therefore it suffices, by definition, to prove Theorem \ref{thm: main2} for $L$-uniform groups.

We recall the link between affinoid enveloping algebras and distribution algebras. The first part of the following version of the Lazard isomorphism is essentially proved in \cite[\S 6.6]{Sch} but stated only in a special case; it was then proven (in somewhat more generality) in \cite[Lemma 5.2]{AW2}. We give a brief sketch of the proof for the convenience of the reader.

\begin{prop}
\label{prop: Lazard}
Assume that $G$ is $L$-uniform. Then there is an isomorphism $\Psi_{G}\, :\, \widehat{U(\mf{g})}_{K} \rightarrow D_{r_{0}}(G,K)$, where $\mf{g}=p^{-1}L_{G}$. It is compatible with morphisms $\alpha\, :\, G \rightarrow H$ in the sense that the diagram
\[
\xymatrix{ \widehat{U(\mf{g})}_{K} \ar[r]^{U(\alpha)} \ar[d]^{\Psi_{G}} & \widehat{U(\mf{h})}_{K} \ar[d]^{\Psi_{H}} \\
D_{r_{0}}(G,K) \ar[r]^{D(\alpha)} & D_{r_{0}}(H,K)
}
\]
commutes where $U(\alpha)$ and $D(\alpha)$ are the natural maps induced by $\alpha$.
\end{prop}

\begin{proof}
We let $\mf{g}_{0}$ denote $\mf{g}$ thought of as a $\Zp$-Lie algebra by forgetting the $\mc{O}_{L}$-structure. Recall that there is a natural morphism $\mc{L}(G_{0}) \rightarrow D_{r_{0}}(G_{0},K)$ factoring through $D(G_{0},K)$ defined by
$$ Xf=\frac{d}{dt}\left( f(\exp(-tX)) \right)\mid_{t=0} $$
for $f\in \mc{C}^{\la}(G_{0},K)$ and $X\in \mc{L}(G_{0})$ (see e.g. \cite[\S 5.2]{Sch}); it gives an inclusion $\mf{g}\rightarrow D_{r_{0}}(G,K)$. By e.g. \cite[Proposition 6.3]{Sch} there is an isomorphism 
$$\Psi_{G_{0}} \, :\, \widehat{U(\mf{g}_{0})}_{K}=\widehat{U(\mf{g}_{0}\otimes_{\Zp}\mc{O}_{L})}_{K} \rightarrow D_{r_{0}}(G_{0},K) $$
which is compatible with the embeddings of $\mf{g}_{0}$ on both sides. Put $\mf{k}=\ker(\mf{g}_{0}\otimes_{\Zp}\mc{O}_{L} \rightarrow \mf{g})$ where the map is given by $a\otimes X \mapsto aX$. Then $D_{r_{0}}(G,K)$ is the quotient of $D_{r_{0}}(G_{0},K)$ by the ideal generated by $\mf{k}$ by \cite[Lemma 5.1]{Sch1}. Similarly, $\widehat{U(\mf{g})}_{K}$ is the quotient of $\widehat{U(\mf{g}_{0}\otimes_{\Zp}\mc{O}_{L})}_{K}$ by the ideal generated by $\mf{k}$ by a straightforward argument: using the PBW filtration one sees that the sequence 
$$ 0 \rightarrow \mf{k}.U(\mf{g}_{0}\otimes_{\Zp}\mc{O}_{L}) \rightarrow U(\mf{g}_{0}\otimes_{\Zp}\mc{O}_{L}) \rightarrow U(\mf{g}) \rightarrow 0 $$
is an exact sequence of finitely generated $U(\mf{g}_{0}\otimes_{\Zp}\mc{O}_{L})$-modules. Now use the Artin-Rees Lemma and tensor with $K$ to conclude. Thus $\Psi_{G_{0}}$ induces the desired isomorphism $\Psi_{G}$ by quotienting out by the ideal generated by $\mf{k}$ on the source and target. Finally, the compatibility with morphisms follows from the functoriality of the morphism $\mc{L}(G_{0}) \rightarrow D_{r_{0}}(G_{0},K)$, which is straightforward to check from the defining formula. 
\end{proof}

With these preparations we may now prove Theorem \ref{thm: main2} (which, we recall, one only needs to show for $L$-uniform groups).

\begin{prop}
Theorem \ref{thm: main2} holds for $L$-uniform groups $G$.
\end{prop}

\begin{proof}
The proof is very similar to that of Proposition \ref{prop: main}. First of all note that without loss of generality $L=K$ by \cite[Lemma 2.6]{AW1}. For the purposes of this proof we put $D:=D(G,L)$ and $D_{n}:=D_{r_{n}}(G,L)$. Let $M$ be a coadmissible left $D$-module and put $N=\Ext^{j}_{D}(M,D)$ with $j=j_{D}(M)$. By \cite[Lemma 8.4]{ST2} we may find an integer $t\geq 0$ such that $N_{n}:=N\otimes_{D}D_{n}=\Ext_{D_{n}}(M_{n},D_{n})\neq 0$ for all $n\geq t$, where $M_{n}:=D_{n}\otimes_{D}M$. Because of the bimodule isomorphisms $D_n \cong D_{r_0}(G^{p^n},L) \otimes_{D(G^{p^n},L)} D$ and $D_n \cong D \otimes_{D(G^{p^n},L)} D_{r_0}(G^{p^n},L)$, we see that $M_n \cong D_{r_0}(G^{p^n},L) \otimes_{D(G^{p^n},L)} M$ and $N_n \cong N \otimes_{D(G^{p^n},L)} D_{r_0}(G^{p^n},L)$ for any $n \geq 0$. Therefore we may, as in the proof of Proposition \ref{prop: main}, replace $G$ by $G^{p^t}$ and without loss of generality assume that $t=0$. 

Pick a finite extension $F/L$ such that $\mc{L}(G)\otimes_{L}F$ is split and let $\mf{g}\subseteq \mc{L}(G)\otimes_{L}F$ be a split $\mc{O}_{F}$-sub-Lie algebra. Write $\mf{h} := p^{-1}L_G$ and pick positive integers $m$ and $n$ such that 
\[p^{n}\mf{h}\otimes_{\mc{O}_{L}}\mc{O}_{F} \subseteq p^{m}\mf{g} \subseteq \mf{h}\otimes_{\mc{O}_{L}}\mc{O}_{F}.\] 
We have that $d_{D}(M)=d_{D_{n}}(M_{n})$ and from here on we argue exactly as in the proof of Proposition \ref{prop: main}, using that $D_{n}=D_{r_{0}}(G^{p^{n}},L)\ast G/G^{p^{n}}=\widehat{U(p^{n}\mf{h})}_{L}\ast G/G^{p^{n}}$ using Proposition \ref{prop: Lazard}. Note that the compatibility statement in Proposition \ref{prop: Lazard} ensures that the diagram that Lemma \ref{lem: sandwich} has to be applied to, namely
\[
\xymatrix{D_{r_0}(G^{p^n},F)\ar[r]\ar[d] & \widehat{U(p^{m}\mathfrak{g})}_{F}\ar[d]\\
D_{r_n}(G,F)\ar[r] & F\otimes_L\widehat{U(\mf{h})}_L
}
\]
is commutative.
\end{proof}

\end{document}